\documentclass[reqno]{amsart}
\usepackage{fancyhdr}
\usepackage[dvips]{color}
\usepackage{mathrsfs}
\usepackage{amsfonts}
\usepackage{graphicx}
\usepackage{xcolor}
\usepackage{amscd}
\usepackage{amsmath}
\usepackage{amssymb}
\usepackage{latexsym}
\usepackage{hyperref}
\usepackage[all]{xy}
\theoremstyle{plain}

\newtheorem{corollary}{\bf Corollary}
\newtheorem{definition}{\bf Definition}
\newtheorem{example}{\bf Example}
\newtheorem{lemma}{\bf Lemma}

\newtheorem{proposition}{\bf Proposition}
\newtheorem{remark}{Remark}

\newtheorem{theorem}{\bf Theorem}
{}
\numberwithin{equation}{section}
\newcommand\tr{\mathrm{tr}}
\newcommand\dv{\mathrm{div}}

\begin{document}

\title[Classification of base of warped product almost Ricci solitons]{Classification of base of warped product almost Ricci solitons}

\author{Jos\'e N. V. Gomes}
\address{(Jos\'e N. V. Gomes) Current: Department of Mathematics, Lehigh University, Bethlehem, PA 18015, USA; Permanent: Departamento de Matem\'atica, Universidade Federal do Amazonas, 69080-900, Manaus-AM-BR}
\email{jnvgomes@pq.cnpq.br, jov217@lehigh.edu}
\urladdr{https://www1.lehigh.edu; http://www.ufam.edu.br}

\author{Manoel V. M. Neto}
\address{(Manoel V. M. Neto) Departamento de Matem\'atica, Universidade Federal do Piau\'i, 64049-550, Teresina-PI-BR}
\email{mvieira@ufpi.edu.br}
\urladdr{http://www.ufpi.edu.br}

\keywords{Conformally flat; Einstein type metrics, Ricci-Hessian type manifolds}

\subjclass[2010]{Primary 53C15, 53A30; Secondary 53C24, 53C12}

\begin{abstract}
In this paper we study a Ricci-Hessian type manifold $(\Bbb{M},g,\varphi,f,\lambda)$ which is closely related to the construction of almost Ricci soliton realized as a warped product. We classify certain classes of the Ricci-Hessian type manifolds and derive some implications for almost Ricci solitons and generalized $m$--quasi-Einstein manifolds. We consider two complementary cases: $\nabla f$ and $\nabla\varphi$ are linearly independent in $C^\infty(\Bbb{M})$--module $\mathfrak{X}(\Bbb{M})$; and $\nabla f=h\nabla\varphi$ for a smooth function $h$ on $\Bbb{M}$. In the first case we show that the vector field $\nabla\lambda$ belongs to the $C^\infty(\Bbb{M})$--module generated by $\nabla f$ and $\nabla\varphi$, while in the second case, under additional hypothesis, the manifold is, around any regular point of $f$, locally isometric to a warped product.
\end{abstract}
\maketitle

\section{Introduction}

The main aim of this paper is to classify the base of almost Ricci solitons realized as warped products. This latter was introduced in \cite{BO} and appear in a natural manner in Riemannian geometry and their applications abound. Almost Ricci solitons were introduced in \cite{prrs}, currently its arise from the Ricci-Bourguignon flow as discovered recently by Catino et al. in \cite{Catino2}. They were motivated by traditional Ricci solitons which correspond to selfsimilar solutions of Ricci flow and often arise as limits of dilations of singularities in the Ricci flow, cf. Hamilton \cite{hamilton2}. These two classes of manifolds are generalizations of Einstein manifolds and has been attracting a lot of attention in the mathematical community. Besides them, new classes of Einstein type manifolds arise in several different frameworks. For instance, Case-Shu-Wei \cite{case1} introduced the concept of $m$--quasi-Einstein manifold, i.e., a Riemannian manifold whose modified Bakry-Emery Ricci tensor is a constant multiple of the metric tensor. We wish to remind the reader that this concept originated from the study of Einstein warped product manifolds, cf. Besse \cite{besse} and Kim-Kim \cite{kim}. Indeed, a necessary condition for a warped product to be an Einstein manifold is its base to be an $m$--quasi-Einstein manifold. Moreover, it was showed in \cite{kim} that a compact Einstein warped product manifold with non-constant warping function does not exist if the scalar curvature is non-positive. Recently, Barros-Batista-Ribeiro \cite{BBR} provided some volume estimates for Einstein warped product manifolds similar to a classical result due to Calabi \cite{Calabi} and Yau \cite{Yau} for complete Riemannian manifolds with non-negative Ricci curvature. For this, they made use of the approach of $m$--quasi-Einstein manifolds. In particular, they also presented an obstruction for the existence of such a class of manifolds. An $m$--quasi-Einstein manifolds also arise of a flow as we shall see ahead. For more details about this subjects see \cite{BGR,br2,case2,GWX,hepeterwylie,oneill}.

Thinking of obtaining new examples of Einstein type manifolds, the next step is the construction of Ricci solitons or almost Ricci solitons that are realized as warped products. By analyzing the recent results in the literature we can notice that some Einstein type manifolds already are related to a warped product. For example, a locally conformally flat gradient almost Ricci soliton with dimension at least three is, around any regular point of the potential function, locally a warped product with fiber of constant sectional curvature. This was proved by Catino in \cite{C} when he introduced the notion of the generalized quasi-Einstein manifold which generalizes concepts of Ricci soliton, almost Ricci soliton and $m$--quasi-Einstein manifold. He called a complete Riemannian manifold $(\Bbb{M}^{n},g)$, $n\geq 2$, a (gradient) generalized quasi-Einstein manifold if there are smooth functions $\psi$, $\lambda$ and $\mu$ on $\Bbb{M}^{n}$ satisfying
\begin{equation}\label{gqem}
Ric+\nabla^{2}\psi-\mu d\psi\otimes d\psi=\lambda g,
\end{equation}
where $Ric$ stands for the Ricci tensor of the Riemannian manifold $(\Bbb{M}^n,g)$. Observe that Catino essentially modified the definition of $m$--quasi-Einstein manifold by adding the condition on the parameter $\lambda$ to be a smooth function as well as introducing a new function $\mu$. In fact, as mentioned earlier an $m$--quasi-Einstein manifold $(\Bbb{M}^n,g,\psi,\lambda)$ is characterized by the equation
\begin{equation}\label{m-gqem}
Ric_m^\psi:=Ric+\nabla^{2}\psi-\frac{1}{m} d\psi\otimes d\psi=\lambda g,
\end{equation}
where $m$ is a positive integer or $m=\infty$, $\lambda$ is a constant and $Ric_m^\psi$ is the modified Bakry-Emery Ricci tensor of $g$.

For $0<m<\infty$ and considering the non-constant function $u=e^{-\frac{\psi}{m}}$, equation \eqref{m-gqem} can be rewritten as
\begin{equation}\label{a5}
Ric-\frac{m}{u}\nabla^{2}u=\lambda g.
\end{equation}
Hence, when $m$ is finite, we can use \eqref{m-gqem} to study \eqref{a5} and vice versa.

A particular case of a generalized quasi-Einstein manifold was well explored by Barros and Ribeiro in \cite{BR}, where they studied the case $\mu=1/m$ in \eqref{gqem}
introducing the concept of a generalized $m$--quasi-Einstein manifold. A good geometric structure shows that the class of compact generalized $m$--quasi-Einstein manifolds with constant scalar curvature is non-trivial and rigid. More precisely, let us consider the smooth functions on the standard unit sphere $(\mathbb{S}^n,g)$, $n\geq2$,
\begin{equation*}
\psi=-m\ln\big(\tau-\frac{h_v}{n}\big)\quad \mbox{and}\quad \lambda=(n-1)-\frac{mnh_v}{n\tau-h_v},
\end{equation*}
where $\tau\in(\frac{1}{n},+\infty)$ is real number and $h_v$ is the height function with respect to some fixed unit vector $v\in\mathbb{R}^{n+1}$. Barros and
Ribeiro showed that the quadruple $(\mathbb{S}^n,g,\psi,\lambda)$ is a generalized $m$--quasi-Einstein manifold. The rigidity of this class was proved by Barros and Gomes in \cite{BG}.

More recently, Catino et al. \cite{CMMR} introduced the notion of Einstein-type manifold or Einstein-type structure on a Riemannian manifold, unifying various particular cases studied in the literature, such as almost Ricci soliton, Yamabe soliton and generalized quasi-Einstein manifold. They showed that these general structures can be locally classified when the Bach tensor is null.

Coming back to the idea of construction of the Einstein type manifolds as warped products, we can say that Bryant's example appears as a prototype for this. He constructed a steady Ricci soliton as the warped product $(0,\infty)\times_f\Bbb{S}^m$, $m>1$, with a radial warping function $f$. Bryant did not himself publish this result, but it can be checked in \cite{ChowEtal}. Other constructions of this kind can be found in \cite{DW,GK,Ivey2}. More generally, a necessary and sufficient condition for constructing a gradient Ricci soliton warped product appear in \cite{RSWP}. This construction originated a class of complete expanding Ricci soliton warped products which has as a fiber an Einstein manifold with non-positive scalar curvature.

The initial idea in \cite{RSWP} was also used in the construction of almost Ricci solitons that are realized as warped product \cite{ARSWP}. As a consequence of the ODE's theory, the techniques used in these two articles are generally different. In both cases the authors also discussed some obstructions to the referred constructions and especially when the base is compact. Moreover, they found in \cite{RSWP,ARSWP} a new class of Einstein type metrics, namely:
\begin{definition}\label{Def1}
A Riemannian manifold $(\Bbb{M}^n,g)$ is a Ricci-Hessian type manifold if there exist real smooth functions $\varphi,\,f,\,\lambda$ on $\Bbb{M}$
satisfying
\begin{equation}\label{EqFund-RHTM}
Ric+\nabla^2\varphi=\lambda g+\frac{m}{f}\nabla^2 f
\end{equation}
where $m$ is a positive integer and $f$ is a positive function.
\end{definition}
We will refer to \eqref{EqFund-RHTM} as the \emph{Ricci-Hessian type equation} or \emph{fundamental equation}, to $f$ as the \emph{warping function}, to $\varphi$ as the \emph{potential function} and to $g$ as \emph{Ricci-Hessian type metric}. For simplicity, we will say that $(\Bbb{M}^n,g,f,\varphi,\lambda)$ is a Ricci-Hessian type manifold.

Ricci-Hessian type manifolds are generalizations of Einstein manifolds. Examples include, gradient almost Ricci solitons (hence, gradient Ricci solitons) and generalized $m$--quasi-Einstein manifolds (hence, $m$--quasi-Einstein manifolds). We emphasize that this new class of manifolds is closely related to the construction of almost Ricci solitons that are realized as warped product, they arise as base of this construction, for further details see \cite{ARSWP}.

We point out that if $\nabla\varphi$ is homothetic vector field and $\lambda$ is constant, equation \eqref{EqFund-RHTM} reduces to one associated $m$--quasi-Einstein metric \eqref{a5}. Ricci-Hessian type metrics with $m=1$ and warping function satisfying $\Delta f+\Lambda f=0$, for some constant $\Lambda$, can be related to static metrics. Interests in static metrics are motivated by general relativity, see for example \cite{besse,corvino}. Moreover, notice that the following relation is true
\begin{equation*}
\nabla^2\ln(f)=\frac{1}{f}\nabla^2f-\frac{1}{f^2}df\otimes df,
\end{equation*}
then equation \eqref{EqFund-RHTM} is equivalent to
\begin{equation}\label{FundEqMRS}
Ric+\nabla^2\eta  = \lambda g + \frac{1}{m}d\xi\otimes d\xi,
\end{equation}
where $\xi=-m\ln(f)$ and $\eta=\varphi+\xi$. Equation \eqref{FundEqMRS} refers to a gradient almost modified Ricci soliton and, when $\lambda$ is constant, to a gradient modified Ricci soliton. These two cases were studied by Freitas Filho in \cite{Airton}. In particular, he showed that a modified Ricci soliton appears as part of a self-similar solution of the modified Harmonic-Ricci flow which results in a new characterization of $m$--quasi-Einstein manifolds.

The purpose of this paper is to classify certain classes of the Ricci-Hessian type manifolds $(\Bbb{M}^n,g,f,\varphi,\lambda)$ and derive some implications for generalized $m$--quasi-Einstein manifolds and gradient almost Ricci solitons. We consider two complementary cases:
\begin{enumerate}
\item[(A)] $\nabla f$ and $\nabla\varphi$ are linearly independent in $C^\infty(\Bbb{M})$--module $\mathfrak{X}(\Bbb{M})$;
\item[(B)] $\nabla f=h\nabla\varphi$ for a smooth function $h$ on $\Bbb{M}$.
\end{enumerate}

In the first case, we will show that the vector field $\nabla\lambda$ belongs to the $C^\infty(\Bbb{M})$--module generated by $\nabla f$ and $\nabla\varphi$, while in the second case, under additional hypothesis, the Riemannian manifold $(\Bbb{M}^n,g)$ is, around any regular point of $f$, locally isometric to a warped product.

Initially, we establish restrictions on the functions $f,\,\varphi$ and $\lambda$ that parametrize a Ricci-Hessian type manifold.

\begin{proposition}\label{PA}
For all Ricci-Hessian type manifolds $(\Bbb{M}^n,g,f,\varphi,\lambda)$ we have
\begin{enumerate}
\item $fd\varphi\wedge d\lambda=df\wedge [(n+m-2)d\lambda+\frac{m}{f}d[g(\nabla\varphi,\nabla f)]-d|\nabla\varphi|^2+2\lambda d\varphi+d\Delta\varphi].$
\item $df\wedge d\varphi\wedge d\lambda=0$.
\end{enumerate}
\end{proposition}

Proposition \ref{PA} allows to deduce the following theorem that resolves the case (A).

\begin{theorem}\label{MainThmCT}
Let $(\Bbb{M}^n,g,f,\varphi,\lambda)$, $n\geq3$, be a Ricci-Hessian type manifold. If $\nabla f$ and $\nabla\varphi$ are linearly independent in $C^\infty(\Bbb{M})$--module $\mathfrak{X}(\Bbb{M})$, then the vector field $\nabla\lambda$ belongs to the $C^\infty(\Bbb{M})$--module generated by $\nabla f$ and $\nabla\varphi$.
\end{theorem}

For what follows an $f$--hypersurface is the submanifold of $\Bbb{M}$ that is given by inverse image of a regular value of $f$.  We say that $(\Bbb{M}^n,g)$ is $\varphi$--radially flat if its curvature tensor satisfies $Rm(\cdot,\nabla\varphi)\nabla\varphi=0$. Also we shall consider the curvature form given by
\begin{equation*}
\Omega(X,Y)(\cdot,\cdot)=g(Rm(\cdot,\cdot)X,Y) \quad \mbox{for all} \quad X,Y\in\mathfrak{X}(\Bbb{M}).
\end{equation*}

In the case (B) we have the following results in the context of the Ricci-Hessian type manifolds with harmonic Weyl tensor.

\begin{proposition}\label{MainThmEP}
Let $(\Bbb{M}^n,g,f,\varphi,\lambda)$, $n\geq3$, be a Ricci-Hessian type manifold with harmonic Weyl tensor. Then
\begin{equation*}
[(m+n-2)Ric(\nabla f,\cdot)-fRic(\nabla\varphi,\cdot)+(n-2)\nabla^2\varphi(\nabla f,\cdot)]\wedge df=(n-1)f\Omega(\nabla\varphi,\nabla f).
\end{equation*}
In particular, the vector field $\nabla f$ is an eigenvector of the Ricci tensor provided that the vector field $\nabla\varphi$ is normal
on $f$--hypersurfaces.
\end{proposition}

Proposition \ref{MainThmEP} allows to deduce the following theorem that resolves the case (B).

\begin{theorem}\label{MainThm22}
Let $(\Bbb{M}^n,g,\varphi,f,\lambda)$, $n\geq 3$, be a locally conformally flat Ricci-Hessian type manifold with $\nabla\varphi$ normal on $f$--hypersurfaces.
In addition, suppose that $(\Bbb{M}^n,g)$ is $\varphi$--radially flat. Then, $(\Bbb{M}^n,g)$ is, around any regular point of $f$, locally isometric to a warped product with $(n-1)$--dimensional Einstein fiber.
\end{theorem}

The paper is organized as follows. In Section \ref{Preliminaries} we fix notation, comment about facts that will be used in our proof and list without proofs all the main formulas which are to be exploited in the course of this work. In Section \ref{Auxiliary results} we derive some general formulas for a Ricci-Hessian type manifold which are to be used for the establishment of the desired results. Immediate applications will be give for generalized $m$--quasi-Einstein manifolds and almost Ricci solitons. We close this section showing Theorem \ref{MainThmCT}. In Section \ref{RHTM HWT} we study Ricci-Hessian type manifolds with harmonic Weyl tensor and as an application we prove Theorem \ref{MainThm22}. This proof is motivated by the corresponding result for gradient Ricci solitons proven in \cite{Cao-Cheng} and for generalized quasi-Einstein manifolds proven in \cite{C}. In our case, it was necessary to be careful with the set of zeroes of the gradient vector field $\nabla\varphi$ that we denote by $\mathcal{Z}_\varphi$ (see Remark \ref{MainRemark}). Besides, we focused on using of two important facts, one of them due to Derdzinski \cite{Derdzinski} and other is the Proposition 16.11 in \cite{besse}. So, two cases was carefully drafted in Proposition \ref{MainThm2}. We point out that the formulas established in Sections \ref{Auxiliary results} and \ref{RHTM HWT} are applicable to Ricci solitons, almost Ricci solitons, $m$--quasi-Einstein manifolds and generalized $m$--quasi-Einstein manifolds. We also make some concluding remarks in Section \ref{Concluding remarks}.

\vspace{0.3cm}
\noindent \textbf{Acknowledgements:}
The authors would like to express their sincere thanks to Dragomir Tsonev and Mikhail Neklyudov for useful comments, discussions and constant encouragement. Jos\'e N.V. Gomes would like to thank the Department of Mathematics-Lehigh University, where part of this work was carried out. He is grateful to Huai-Dong Cao for the warm hospitality and his constant encouragement. This work has been partially supported by CNPq, Conselho Nacional de Desenvolvimento Cient\'ifico e Tecnol\'ogico-Brasil.

\section{Preliminaries}\label{Preliminaries}
Let $\Bbb{M}^n\times_f\Bbb{F}^m$ be a warped product of two Riemannian manifolds with warped metric $g=g_{\Bbb{M}}+f^2g_{\Bbb{F}}$. Suppose that, for some smooth functions $\varphi$ and $\lambda$ on $\Bbb{M}$, $(\Bbb{M}^n\times_f\Bbb{F}^m,\nabla\tilde\varphi,\tilde\lambda)$ is a gradient almost Ricci soliton, i.e., its Ricci tensor satisfies \begin{equation*}
Ric+\nabla^2\tilde\varphi=\tilde\lambda g,
\end{equation*}
where $\tilde\varphi$ and $\tilde\lambda$ are the lift of the smooth functions $\varphi$ and $\lambda$ on $\Bbb{M}$ to $\Bbb{M}^n\times_f\Bbb{F}^m$, respectively. On the other hand, it is known that the Ricci tensor of warped product $\Bbb{M}^n\times_f\Bbb{F}^m$ satisfies
\begin{equation*}
Ric=Ric_\Bbb{M}-\frac{m}{f}\nabla^2f.
\end{equation*}
Then, on the base $\Bbb{M}$ holds true
\begin{equation*}
Ric_\Bbb{M}+\nabla^2\varphi=\lambda g_{\Bbb{M}}+\frac{m}{f}\nabla^2f,
\end{equation*}
which motivates the equation mentioned in the Definition \ref{Def1}. In our context, we emphasize that the notion of Ricci-Hessian type
manifolds is natural not only because of the latter equation but also for the fact that a Ricci-Hessian type manifold, under some additional conditions
for $f$, $\varphi$ and $\lambda$, is the base of a gradient almost Ricci soliton warped product with $m$--dimensional fiber (which is necessarily an Einstein manifold), \emph{warping function} $f$, \emph{potential
function} $\tilde\varphi$ and \emph{soliton function} $\tilde\lambda$, for further details see \cite{ARSWP}.

The following example shows that the standard sphere and the hyperbolic space both possess the Ricci-Hessian type structure.

\begin{example}[\cite{RSWP}]\label{ExRHTM}
Let $(\Bbb{M}^n(\tau),g_{\circ})$ be the standard sphere $\Bbb{S}^n$ or the hyperbolic space $\Bbb{H}^n$ for $\tau=1$ or $\tau=-1$ respectively. We denote by $h_v$ the height function with respect to a fixed unit vector $v\in\Bbb{R}^{n+1}$. Then for each real number $m\neq0$, the functions $\lambda=\tau(n-1) -\frac{\tau}{m}h_v^2- h_v$, $f=e^{-\frac{\tau}{m}h_v}$ and $\varphi=\frac{1}{2m}h_v^2$ satisfy equation \eqref{EqFund-RHTM} on $(\Bbb{M}^n(\tau),g_{\circ})$. Indeed, we get
\begin{equation*}
\nabla^2\varphi=\frac{1}{m}dh_v\otimes dh_v-\frac{\tau}{m}h_v^2g_{\circ},\quad\frac{m}{f}\nabla^2f=\frac{1}{m}dh_v\otimes dh_v+ h_vg_{\circ},\quad  Ric=\tau(n-1)g_{\circ}
\end{equation*}
and the result follows.
\end{example}

\begin{remark}[\cite{RSWP}]\label{remark1}
Its worth mentioning that the class of generalized $m$--quasi-Einstein metrics is a subset of the class of the Ricci-Hessian type metrics. In fact, for a generalized $m$--quasi-Einstein manifold $(\Bbb{M}^n,g,\psi,\lambda)$ satisfying \eqref{m-gqem} we can take $m=4r$, $\varphi=\frac{\psi}{2}$ and $f=e^{-\frac{\varphi}{r}}$ to deduce that $(\Bbb{M}^n,g,\varphi,f,\lambda)$ satisfies a Ricci-Hessian type equation. In contrast, if we assume that $(\Bbb{M}^n,g,\varphi,f,\lambda)$ satisfies \eqref{EqFund-RHTM}, then by a straightforward computation we have
\begin{equation}\label{Eq3AuxEx1}
Ric+\nabla^2\eta  = \lambda g + \frac{1}{m}d\xi\otimes d\xi
\end{equation}
where $\xi=-m\ln(f)$ and $\eta=\varphi+\xi$ (note that $\eta\neq\xi$ in \eqref{Eq3AuxEx1}).
\end{remark}

From now on, $\nabla$ will stand for the Levi-Civita connection on a Riemannian manifold $(\Bbb{M}^{n},g)$ with Riemann curvature tensor $Rm$ given by
\begin{equation*}
Rm(X,Y)Z=[\nabla_X,\nabla_Y]Z-\nabla_{[X,Y]}Z.
\end{equation*}

The \textit{Weyl curvature tensor} $W$ of $(\Bbb{M}^n,g)$ is defined by the following decomposition formula
\begin{equation}\label{c1}
W=Rm-S\odot g,
\end{equation}
where $S$ is the \textit{Schouten tensor} defined by
\begin{equation}\label{c2}
S=\frac{1}{n-2}\Big(Ric-\frac{R}{2(n-1)}g\Big)
\end{equation}
and $S\odot g$ is the \textit{Kulkarni-Nomizu product} given by
\begin{eqnarray*}
&&S\odot g(X,Y,Z)\\
&=&\frac{1}{n-2}\big(Ric(X,Z)Y+g(X,Z)Ric(Y)-g(Y,Z)Ric(X)-Ric(Y,Z)X\big)\\
&&-\frac{R}{(n-1)(n-2)}\big(g(X,Z)Y-g(Y,Z)X\big),
\end{eqnarray*}
for $Ric(\cdot,\cdot)=g(Ric(\cdot),\cdot)$ and $R$ denoting the trace of $Ric$.

A Riemannian manifold $(\Bbb{M}^{n},g)$ is \textit{locally conformally flat} if each point of $\Bbb{M}$ lies in a neighborhood which is conformally diffeomorphic to an open subset of Euclidean space $\Bbb{R}^{n}$ with the canonical metric $g_\circ$, i.e., if there is a diffeomorphism $\Psi :V\subset \Bbb{R}^{n}\to U\subset\Bbb{M}^n$ such that $\Psi^{*}g=f^{2}g_\circ$ for some positive smooth function $f$ and any open subset $U$ in $\Bbb{M}$. All surfaces are conformally flat because they admit isothermal coordinates. It is well-known that every Riemannian manifold with constant sectional curvature is locally conformally flat, but the converse is not true.  More precisely, a locally conformally flat manifold has constant sectional curvature if and only if it is an Einstein manifold. The class of the locally conformally flat Riemannian manifolds has a classical characterization of Schouten in terms of the Weyl tensor as follows. A Riemannian manifold $(\Bbb{M}^{n},g)$, $n\geq3$, is conformally flat if and only if $W=0$ and the Schouten tensor $S$ is Codazzi, that is, for all $X,Y,Z\in\mathfrak{X}(\Bbb{M})$
\begin{equation}\label{e1}
(\nabla_{X}S)(Y,Z)=(\nabla_{Y}S)(X,Z).
\end{equation}
As long as that the \textit{Cotton tensor} is denoted by
\begin{equation}\label{CottonTensor}
C(X,Y,Z)=(\nabla_{X}S)(Y,Z)-(\nabla_{Y}S)(X,Z).
\end{equation}

The divergence of a $(1,r)$--tensor $T$ in $(\Bbb{M},g)$ is defined as the $(0,r)$--tensor
\begin{equation*}
(\dv T)(p) = \tr\big(v \mapsto (\nabla_v T)(p)\big),
\end{equation*}
where $p\in \Bbb{M}$, $v\in T_p\Bbb{M},$ $\nabla$ stands for the covariant derivative of $T$ and $\mathrm{tr}$ is the trace calculated in the metric $g.$ In particular, it is known that
\begin{equation}\label{divW}
\dv W = -\frac{n-3}{n-2}C.
\end{equation}

A Riemannian metric has \textit{harmonic Weyl tensor} if the divergence of $W$ vanishes. In dimension three this condition is equivalent to local
conformal flatness. In dimension $n\geq4$, the harmonic Weyl tensor is a weaker condition since local conformal flatness is equivalent to the vanishing of the Weyl tensor. We recall that the Weyl tensor is the main invariant under conformal changes.

We shall consider the wedge product for $1$--forms $\alpha$ and $\beta$ given by determinant convention
\begin{equation*}
\alpha\wedge\beta=\alpha\otimes\beta-\beta\otimes\alpha.
\end{equation*}

In this paper we will be constantly using  the identification of a $(0,2)$--tensor $T$ with its associated $(1,1)$--tensor by the equation
\begin{equation*}
g(TX,Y)=T(X,Y).
\end{equation*}
Thus, we get
\begin{equation*}\label{eq}
\dv(\phi T)=\phi \dv T+T(\nabla\phi,\cdot)\quad \mbox{and}\quad \nabla(\phi T)=\phi\nabla T+d\phi\otimes T
\end{equation*}
for all $\phi\in C^{\infty}(\Bbb{M})$. In particular, we have $\dv(\phi g)=d\phi$. The next identity will be crucial:
\begin{equation}\label{2-IDENT}
2\dv Ric=dR.
\end{equation}
It is known as twice contracted second Bianchi identity. Moreover, by a straightforward computation, we have
\begin{equation}\label{dPI}
d[g(\nabla\phi,\nabla\psi)]=\nabla^2\phi(\nabla\psi,\cdot)+\nabla^2\psi(\nabla\phi,\cdot)
\end{equation}
and
\begin{equation}\label{eq3}
(\nabla_X\nabla^2\phi)Y-(\nabla_Y\nabla^2\phi)X=Rm(X,Y)\nabla\phi
\end{equation}
for all $\phi,\psi\in C^{\infty}(\Bbb{M})$. These identities will be used without further comments.

\section{Auxiliary results and proof of Theorem \ref{MainThmCT}}\label{Auxiliary results}

In this section we shall present some properties which will be useful for the establishment of the desired results. Let us deduce, for a Ricci-Hessian type manifold, general formulas that apply to Ricci solitons, almost Ricci solitons, $m$--quasi-Einstein manifolds and generalized $m$--quasi-Einstein manifolds.

\begin{lemma}\label{LA}
If $(\Bbb{M},g,f,\varphi,\lambda)$ is a Ricci-Hessian type manifold, then the following relation holds
\begin{align}\label{EQ-LA}
\nonumber(\nabla_X Ric)Y-(\nabla_Y Ric)X=&X(\lambda)Y-Y(\lambda)X-\frac{m}{f^2}\{X(f)\nabla^2f(Y)-Y(f)\nabla^2f(X)\}\\
&+\frac{m}{f}Rm(X,Y)\nabla f-Rm(X,Y)\nabla\varphi.
\end{align}
\end{lemma}
\begin{proof}
By covariant derivative of the fundamental equation, we get
\begin{eqnarray}\label{Aux-EQ-LA}
\nonumber&&(\nabla_X Ric)Y-(\nabla_Y Ric)X\\
&=&X(\lambda)Y-Y(\lambda)X-\frac{m}{f^2}\{X(f)\nabla^2f(Y)-Y(f)\nabla^2f(X)\}\\
\nonumber&&+\frac{m}{f}\{(\nabla_X\nabla^2f)Y-(\nabla_Y\nabla^2f)X\}-\{(\nabla_X\nabla^2\varphi)Y-(\nabla_Y\nabla^2\varphi)X\}.
\end{eqnarray}
Using identity \eqref{eq3} in \eqref{Aux-EQ-LA} we obtain \eqref{EQ-LA}.
\end{proof}

\begin{lemma}\label{lem2}
For all Ricci-Hessian type manifold $(\Bbb{M}^{n},g,\varphi,f,\lambda)$ we have the following equivalent equations:
\begin{eqnarray}\label{Eq1lem2}
\nonumber\frac{dR}{2}&=&(n-1)d\lambda-\frac{R-(n-1)\lambda}{f}df -\frac{m-1}{f}Ric(\nabla f,\cdot) + Ric(\nabla\varphi,\cdot)\\
&&+\frac{1}{f}\nabla^2\varphi(\nabla f,\cdot)-\frac{\Delta\varphi}{f}df.
\end{eqnarray}
and
\begin{eqnarray}\label{Eq2lem2}
\nonumber\frac{dR}{2}&=&(n-1)d\lambda+\frac{m}{f}d[g(\nabla\varphi,\nabla f)]-\frac{m(m-1)}{2f^2}d|\nabla f|^2-\frac{d|\nabla\varphi|^2}{2}-\frac{m}{f^2}\Delta f df\\
&&-\frac{m}{f}\lambda df+\lambda d\varphi.
\end{eqnarray}
\end{lemma}
\begin{proof}
Given a point $p\in\Bbb{M}$, consider $\{E_1,\ldots,E_n\}$ a base for $T_p\Bbb{M}$. Take the inner product of \eqref{EQ-LA} with a vector field $E_i$ and make $Y=E_i$.
Then,
\begin{eqnarray*}
&&(\nabla_X Ric)(E_i,E_i)-(\nabla_{E_i} Ric)(X,E_i)\\
&=&X(\lambda)-E_i(\lambda)g(X,E_i)-\frac{m}{f^2}\{X(f)\nabla^2f(E_i,E_i)-E_i(f)\nabla^2f(X,E_i)\}\\
&&+\frac{m}{f}g(Rm(X,E_i)\nabla f,E_i)-g(Rm(X,E_i)\nabla\varphi,E_i).
\end{eqnarray*}
Taking the sum in $i$ we have
\begin{eqnarray}\label{AuxEq2lem2}
\nonumber\tr(\nabla_XRic)-(\dv Ric)X&=&(n-1)X(\lambda)-\frac{m}{f^2}\{X(f)\Delta f-\nabla^2f(\nabla f,X)\}\\
&&-\frac{m}{f}Ric(\nabla f,X)+Ric(\nabla\varphi,X).
\end{eqnarray}
Since the trace commutes with covariant derivative, we use \eqref{2-IDENT} to assert
\begin{equation*}
\tr(\nabla_XRic)-(\dv Ric)X=\frac{1}{2}X(R).
\end{equation*}
Hence, \eqref{AuxEq2lem2} turns
\begin{eqnarray}\label{AuxEq2lem2-2}
\frac{1}{2}X(R)&=&(n-1)X(\lambda)-\frac{m}{f^2}\{X(f)\Delta f-\nabla^2f(\nabla f,X)\}\\
\nonumber&&-\frac{m}{f}Ric(\nabla f,X)+Ric(\nabla\varphi,X).
\end{eqnarray}
Being \eqref{AuxEq2lem2-2} valid for all $X\in\mathfrak{X}(\Bbb{\Bbb{M}})$, it is equivalent to
\begin{equation}\label{EQLEM}
\frac{dR}{2}=(n-1)d\lambda-\frac{m}{f^2}\Delta f df+\frac{m}{f^2}\nabla^2f(\nabla f,\cdot)-\frac{m}{f}Ric(\nabla f,\cdot)+Ric(\nabla\varphi,\cdot).
\end{equation}
By fundamental equation
\begin{equation}\label{Aux-EQLEM}
\frac{m}{f}\nabla^2f(\nabla f,\cdot)=Ric(\nabla f,\cdot)+\nabla^2\varphi(\nabla f,\cdot)-\lambda df \quad\mbox{and}\quad
\frac{m}{f}\Delta f=R+\Delta\varphi-n\lambda.
\end{equation}
Replacing \eqref{Aux-EQLEM} in \eqref{EQLEM} we obtain \eqref{Eq1lem2}.
On the other hand, substituting $Ric(\nabla f,\cdot)$ and $Ric(\nabla\varphi,\cdot)$ in \eqref{EQLEM} for their identifications given by fundamental equation and using identity \eqref{dPI} we deduce  equation \eqref{Eq2lem2}.
\end{proof}

\subsection{Proof of Proposition \ref{PA}}

\begin{proof}
Multiplying equation \eqref{Eq2lem2} by $f^2$ and applying the exterior derivative we get
\begin{align*}
f df\wedge dR=&2(n-1)f df\wedge d\lambda+m df\wedge d[g(\nabla\varphi,\nabla f)]-f df\wedge d|\nabla\varphi|^2-m d\Delta f\wedge df\\
&- mf d\lambda\wedge df +f^2d\lambda\wedge d\varphi+2\lambda f df\wedge d\varphi.
\end{align*}
Now, dividing by $f$ and regrouping some terms we have
\begin{align}\label{df Ext dR}
\nonumber df\wedge dR=&(2n+m-2)df\wedge d\lambda+\frac{m}{f}df\wedge d[g(\nabla\varphi,\nabla f)]-df\wedge d|\nabla\varphi|^2+f d\lambda\wedge d\varphi\\
&+2\lambda df\wedge d\varphi+\frac{m}{f}df\wedge d\Delta f.
\end{align}
The substitution of
\begin{align*}
\frac{m}{f}d\Delta f= dR+d\Delta\varphi+\frac{m}{f^2}\Delta f df-nd\lambda
\end{align*}
in the last term of \eqref{df Ext dR} yields the first part of the proposition, i.e,
\begin{equation}\label{aux1}
fd\varphi\wedge d\lambda=df\wedge[(n+m-2)d\lambda+\frac{m}{f}d[g(\nabla\varphi,\nabla f)]-d|\nabla\varphi|^2+2\lambda d\varphi+d\Delta\varphi]
\end{equation}
The exterior derivative of \eqref{aux1} gives the second part.
\end{proof}

\begin{corollary}\label{gener}
For all generalized $m$--quasi-Einstein manifolds $(\Bbb{M}^n,g,\psi,\lambda)$, is valid
\begin{equation}\label{EQ-uf}
(n+m-2)d\lambda\wedge du=-\frac{u}{m}(n+m-2)d\lambda\wedge d\psi=0,
\end{equation}
where $u=e^{-\frac{\psi}{m}}$. In particular, for $n\geq2$, $\nabla\lambda=\eta\nabla\psi$, for some function $\eta$ on $\Bbb{M}^n$.
\end{corollary}
\begin{proof}
Taking $u=e^{-\frac{\psi}{m}}$, equation \eqref{m-gqem} of $(\Bbb{M}^n,g,\psi,\lambda)$ can be rewritten as \eqref{a5}. Hence, considering $\varphi$ to be constant in item $(1)$ of Proposition \ref{PA}, we obtain \eqref{EQ-uf}. Since $n\geq2$, we have $d\lambda\wedge d\psi=0$, thus we conclude our statement.
\end{proof}

We point out that Corollary \ref{gener} is a extension of the almost Ricci soliton case proved in \cite{prrs} to the generalized $m$--quasi-Einstein manifolds case. This last case was demonstrated very recently in \cite{HLX2} by using a technique analogous to \cite{prrs}, but with the restrictions $n\geq3$ and $m<\infty$. Due to this last restriction, the result in \cite{HLX2} is not applied to almost Ricci soliton. However, our technique allows to obtain the result in \cite{prrs} as proved below.

\begin{corollary}[\cite{prrs}]\label{QRS-gener}
For all almost Ricci soliton $(\Bbb{M}^n,g,\nabla\varphi,\lambda)$, is valid
\begin{equation}\label{DepARS}
d\varphi\wedge d\lambda=0.
\end{equation}
In particular, $\nabla\lambda=\eta\nabla\varphi$ for some function $\eta$ on $\Bbb{M}^n$.
\end{corollary}
\begin{proof}
Observe we always can complete the fundamental equation of an almost Ricci soliton $(\Bbb{M}^n,g,\nabla\varphi,\lambda)$ with a constant positive function $f$, such that we can apply the first item of Proposition \ref{PA} in order to obtain \eqref{DepARS}, which is sufficient to conclude the corollary.
\end{proof}

\subsection{Proof of Theorem \ref{MainThmCT}}
\begin{proof}
The result immediately follows from part two of Proposition \ref{PA} and the assumption of linear independence of the vector fields $\nabla f$ and $\nabla\varphi$.
\end{proof}

\section{Ricci-Hessian type manifold with harmonic Weyl tensor}\label{RHTM HWT}

In this section we prove properties concerning to harmonicity condition of the Weyl tensor.

\begin{lemma}\label{lem3}
Let $(\Bbb{M}^{n},g,\varphi,f,\lambda)$ be a Ricci-Hessian type manifold. If $(\Bbb{M}^{n},g)$, $n\geq3$, has harmonic Weyl tensor, then
\begin{align}\label{EQLem3}
\nonumber Rm(X,Y)\nabla f=&\frac{1}{f}\{X(f)\nabla^2f(Y)-Y(f)\nabla^2f(X)\}+\frac{f}{2m(n-1)}\{X(R)Y-Y(R)X\}\\
&+\frac{f}{m}\{Rm(X,Y)\nabla\varphi+Y(\lambda)X-X(\lambda)Y\}.
\end{align}
\end{lemma}

\begin{proof}
For $n>3$ and $\dv W=0$, from \eqref{CottonTensor} and \eqref{divW} we have that $S$ is Codazzi. For $n=3$ this condition is equivalent to local conformal flatness. By \eqref{c2} a direct computation yields
\begin{equation}\label{e1}
(\nabla_XRic)Y-(\nabla_YRic)X=\frac{1}{2(n-1)}(X(R)Y-Y(R)X).
\end{equation}
From Lemma \ref{LA} and \eqref{e1} we obtain
\begin{eqnarray*}
&&\frac{1}{2(n-1)}(X(R)Y-Y(R)X)+Rm(X,Y)\nabla\varphi\\
&=&X(\lambda)Y-Y(\lambda)X+\frac{m}{f}(Rm(X,Y)\nabla f)-\frac{m}{f^2}\{X(f)\nabla^2f(Y)-Y(f)\nabla^2f(X)\}.
\end{eqnarray*}
Multiplying the last equation by $\frac{f}{m}$ and reordering its terms we get the desired formula.
\end{proof}

We denote $\mathcal{Z}_\varphi$ the set of zeroes of the gradient vector field $\nabla\varphi$ of the potential function $\varphi$ of an Ricci-Hessian type manifold $(\Bbb{M}^{n},g,\varphi,f,\lambda)$. Since $\mathcal{Z}_\varphi$ can be extremely large, we will assume from now on that it can be at most a discrete subset of $\Bbb{M}$. For instance, it is well known that non-trivial gradient conformal vector fields has at most two zeroes (see \cite{Tashiro}). However, it is noteworthy that the existence of a gradient conformal vector fields on a Riemannian manifold imply geometric restrictions on the manifold.

\begin{remark}\label{MainRemark}
As already mentioned, a Ricci-Hessian type manifold can becomes a generalized $m$--quasi-Einstein manifold if we take $\varphi$ constant. But, differently of the previous section, here this is not convenient for us, once that $\mathcal{Z}_\varphi$ would be the whole manifold $\Bbb{M}$. However, we always can convert a generalized $m$--quasi-Einstein manifold $(\Bbb{M},g,\psi,\lambda)$ to a Ricci-Hessian type manifold $(\Bbb{M},g,\varphi,f,\lambda)$, such that the set of the critical points of the respective potential functions coincide. For instance, in Remark \ref{remark1} we have
\begin{equation*}
\varphi=\frac{\psi}{2},\quad f=e^{-\frac{\psi}{2r}}\quad \mbox{and}\quad \frac{1}{2}\nabla\psi=\nabla\varphi=-\frac{r}{f}\nabla f.
\end{equation*}
Assuming that $\mathcal{Z}_\psi$ is a discrete subset of $\Bbb{M}$, the following results will be a natural generalization of the $m$--quasi Einstein manifold case to Ricci-Hessian type manifold case.
\end{remark}

\subsection{Proof of Proposition \ref{MainThmEP}}
\begin{proof}
Effecting the inner product of \eqref{EQLem3} with $\nabla f$ we have
\begin{eqnarray}\label{p1}
\nonumber &&Y(f)\nabla^2f(\nabla f,X)-X(f)\nabla^2f(\nabla f,Y)\\
\nonumber&=&\frac{f^2}{m}g(Rm(X,Y)\nabla\varphi,\nabla f)-\frac{f^2}{m}\Big\{d\lambda(X)-\frac{dR(X)}{2(n-1)}\Big\}df(Y)\\
&&+\frac{f^2}{m}\Big\{d\lambda(Y)-\frac{dR(Y)}{2(n-1)}\Big\}df(X).
\end{eqnarray}
Follows of Lemma \ref{lem2} that
\begin{align}\label{p2}
\nonumber d\lambda-\frac{dR}{2(n-1)}=&\frac{m-1}{(n-1)f}Ric(\nabla f,\cdot)+\frac{R-(n-1)\lambda}{(n-1)f}df-\frac{1}{(n-1)}Ric(\nabla\varphi,\cdot)\\
&-\frac{1}{(n-1)f}\nabla^2\varphi(\nabla f,\cdot)+\frac{\Delta\varphi}{(n-1)f}df.
\end{align}
Putting \eqref{p2} in \eqref{p1} we obtain
\begin{align}\label{p3}
\nonumber&Y(f)\nabla^2f(\nabla f,X)-X(f)\nabla^2f(\nabla f,Y)\\
\nonumber=&\frac{f^2}{m}g(Rm(X,Y)\nabla\varphi,\nabla f)-\frac{f(m-1)}{m(n-1)}\{Ric(\nabla f,X)df(Y)-Ric(\nabla f,Y)df(X)\}\\
&+\frac{f^2}{m(n-1)}\{Ric(\nabla\varphi,X)df(Y)-Ric(\nabla\varphi,Y)df(X)\}\\
\nonumber&+\frac{f}{m(n-1)}\{\nabla^2\varphi(\nabla f,X)df(Y)-\nabla^2\varphi(\nabla f,Y)df(X)\}.
\end{align}
On the other hand, the fundamental equation \eqref{EqFund-RHTM} yields
\begin{eqnarray}\label{p4}
\nonumber&&Y(f)\nabla^2f(\nabla f,X)-X(f)\nabla^2f(\nabla f,Y)\\
\nonumber&=&\frac{f}{m}\{Ric(\nabla f,X)df(Y)-Ric(\nabla f,Y)df(X)\}\\
&&+\frac{f}{m}\{\nabla^2\varphi(\nabla f,X)df(Y)-\nabla^2\varphi(\nabla f,Y)df(X)\}.
\end{eqnarray}
From \eqref{p3} and \eqref{p4} we get
\begin{eqnarray*}
&&(m+n-2)\{Ric(\nabla f,X)df(Y)-Ric(\nabla f,Y)df(X)\}\\
&=&(n-1)f g(Rm(X,Y)\nabla\varphi,\nabla f)+f\{Ric(\nabla\varphi,X)df(Y)-Ric(\nabla\varphi,Y)df(X)\}\\
&&-(n-2)\{\nabla^2\varphi(\nabla f,X)df(Y)-\nabla^2\varphi(\nabla f,Y)df(X)\}.
\end{eqnarray*}
Now, the desired result follows from definitions of the curvature form and the wedge product.

To the particular case $\nabla\varphi$ be normal on $f$--hypersurfaces, we have $\nabla\varphi=h\nabla f$ for some smooth function $h$ on $\Bbb{M}$. Thus,
\begin{equation}\label{b1}
\nabla^2\varphi(\nabla f,\cdot)=h\nabla^2f(\nabla f,\cdot)+\nabla f(h)df.
\end{equation}
Using \eqref{b1} and \eqref{EqFund-RHTM} we obtain
\begin{equation}\label{eqAuxCart}
Ric(\nabla f)=\frac{m-fh}{f}\nabla^2f(\nabla f)+(\lambda-\nabla f(h))\nabla f.
\end{equation}
Replacing \eqref{b1} in Proposition \ref{MainThmEP} we get
\begin{equation}
\big[(m+n-2)Ric(\nabla f,\cdot) - fhRic(\nabla f,\cdot)+(n-2)\nabla^2f(\nabla\varphi,\cdot)\big]\wedge df=0.
\end{equation}
Hence, for all $p\in \Bbb{M}$ such that $h(p)\neq0$, we have
\begin{equation*}
\big[\nabla^2f(\nabla f,\cdot)-\frac{fh-(m+n-2)}{(n-2)h}Ric(\nabla f,\cdot)\big]\wedge df=0,
\end{equation*}
It follows that in a neighborhood of $p$ there exists a smooth function $\mu$ such that
\begin{equation}\label{b2}
\nabla^2f(\nabla f)=\frac{fh-(m+n-2)}{(n-2)h}Ric(\nabla f)+\mu\nabla f.
\end{equation}
Substituting \eqref{b2} in \eqref{eqAuxCart} we get
\begin{equation}\label{b3}
Ric(\nabla f)=\frac{(m-fh)\mu+(\lambda-\nabla f(h))f}{(m-fh)^2+m(n-2)}(n-2)h\nabla f.
\end{equation}
Since the points where $h$ is zero belong to the set $\mathcal{Z}_\varphi$, it follows by an argument of continuity that equation \eqref{b3} also applies to these
points, which is sufficient to complete the proof of the proposition.
\end{proof}

\begin{proposition}\label{prop1}
Let $(\Bbb{M}^{n},g,\varphi,f,\lambda)$ be a Ricci-Hessian type manifold. If $(\Bbb{M}^{n},g)$, $n\geq3$, has harmonic Weyl tensor and $W(\nabla f,\cdot,\cdot)$ is identically null, then on any $f$--hypersurfaces of $\Bbb{M}$ the following relation holds
\begin{eqnarray}\label{eqprop1}
\nonumber&&(n+m-2)\Big\{\nabla^2f(E_{i},E_{j})-\frac{1}{(n-1)}(\Delta f-\nabla^2f(\nu,\nu))g_{ij}\Big\}\\
&=&-f\Big\{\nabla^2\varphi(E_{i},E_{j})-\frac{1}{(n-1)}(\Delta\varphi-\nabla^2\varphi(\nu,\nu))g_{ij}\Big\}\\
\nonumber&&+\frac{(n-2)f^{2}}{m|\nabla f|^2}\Big\{g(Rm(E_i,\nabla f)\nabla\varphi,E_j)-\frac{Ric(\nabla f,\nabla\varphi)}{(n-1)}g_{ij}\Big\},
\end{eqnarray}
where $\{E_1,\ldots,E_{n-1},\nu= \frac{\nabla f}{|\nabla f|}\}$ is a local orthonormal frame of $\Bbb{M}$ adapted to \linebreak $f$--hypersurfaces.
\end{proposition}
\begin{proof}
Since $W(\nabla f,\cdot,\cdot)$ is identically null, from \eqref{c1} we obtain
\begin{eqnarray*}
Rm(X,Y,Z,\nabla f)=S\odot g (X,Y,Z,\nabla f).
\end{eqnarray*}
Then, by Lemma \ref{lem3} we deduce the following expression
\begin{align*}
&\frac{1}{f}\{g(\nabla f,Y)\nabla^2f(X,Z)-g(\nabla f,X)\nabla^2f(Y,Z)\}+\frac{f}{m}g(Rm(X,Y)Z,\nabla\varphi)\\
&+\frac{f}{m}\{g(\nabla\lambda,X)g(Y,Z)-g(\nabla\lambda,Y)g(X,Z)\}\\
&-\frac{f}{2m(n-1)}\{g(\nabla R,X)g(Y,Z)-g(\nabla R,Y)g(X,Z)\}\\
=&S(X,\nabla f)g(Y,Z)+S(Y,Z)g(X,\nabla f)-S(X,Z)g(Y,\nabla f)-S(Y,\nabla f)g(X,Z).
\end{align*}
Consider  a orthonormal frame $\{E_{1},\ldots,E_{n-1},E_{n}=\nu\}$ for $\nu=\frac{\nabla f}{|\nabla f|}$ around any regular point of $f$. Making $X=\nu$, $Y=E_i$
and $Z=E_j$, and dividing both sides of the last equation by $|\nabla f|$, we get
\begin{eqnarray}\label{t1}
\nabla^2f(E_{i},E_{j})&=&\frac{f^{2}}{m|\nabla f|}g\big(\nabla\lambda -\frac{\nabla R}{2(n-1)},\nu\big)g_{ij}+\frac{f^{2}}{m|\nabla f|}g(Rm(\nu,E_i)E_j,\nabla\varphi)\nonumber\\
&&-f\{S(\nu,\nu)g_{ij}-S(E_{i},E_{j})\},
\end{eqnarray}
where $g_{ij}=g(E_i,E_j)$ for all $i,j\in\{1,\ldots,n-1\}$.

Using equation \eqref{EqFund-RHTM} and the identification for the Schouten tensor, we have
\begin{equation}\label{t2}
S(\nu,\nu)g_{ij}=\frac{1}{n-2}\Big\{\lambda g_{ij}+\frac{m}{f}\nabla^2f(\nu,\nu)g_{ij}-\nabla^2\varphi(\nu,\nu)g_{ij}-\frac{R}{2(n-1)}g_{ij}\Big\}
\end{equation}
and
\begin{equation}\label{t2-2}
S(E_i,E_j)=\frac{1}{n-2}\Big\{\lambda g_{ij}+\frac{m}{f}\nabla^2f(E_i,E_j)-\nabla^2\varphi(E_i,E_j)-\frac{R}{2(n-1)}g_{ij}\Big\}.
\end{equation}
Now, consider equation \eqref{p2} reformulated as
\begin{eqnarray}\label{t3-2}
\nonumber\nabla\lambda-\frac{\nabla R}{2(n-1)}&=&\frac{m-1}{(n-1)f}Ric(\nabla f)+\frac{R-(n-1)\lambda}{(n-1)f}\nabla f-\frac{Ric(\nabla\varphi)}{(n-1)}\\
&&-\frac{\nabla^2\varphi(\nabla f)}{(n-1)f}+\frac{\Delta\varphi\nabla f}{(n-1)f}.
\end{eqnarray}
Combining \eqref{t3-2} and the identity
\begin{equation*}
Ric(\nu,\nu)=\lambda+\frac{m}{f}\nabla^2f(\nu,\nu)-\nabla^2\varphi(\nu,\nu)
\end{equation*}
we get
\begin{align}\label{t4}
\nonumber&\frac{f^2}{|\nabla f|}g\big(\nabla\lambda-\frac{\nabla R}{2(n-1)},\nu\big)\\
\nonumber=&\frac{(m-1)f}{(n-1)} Ric(\nu,\nu)+\frac{(R-(n-1)\lambda)f}{(n-1)}-\frac{f^2Ric(\nabla\varphi,\nu)}{(n-1)|\nabla f|}-\frac{f\nabla^2\varphi(\nu,\nu)}{(n-1)}+\frac{f\Delta\varphi}{(n-1)}\\
\nonumber=&\frac{m(m-1)}{(n-1)}\nabla^2f(\nu,\nu)-\frac{mf}{(n-1)}\nabla^2\varphi(\nu,\nu)+\frac{(R+(m-n)\lambda) f}{(n-1)}+\frac{f\Delta\varphi}{(n-1)}\\
&-\frac{f^2Ric(\nabla\varphi,\nu)}{(n-1)|\nabla u|}.
\end{align}
Substituting \eqref{t2}, \eqref{t2-2} and \eqref{t4}  in \eqref{t1} yields
\begin{align}\label{SUBST}
\nonumber&\frac{n+m-2}{n-2}\nabla^2f(E_{i},E_{j})\\
\nonumber=&-\frac{n+m-2}{(n-1)(n-2)}\nabla^2f(\nu,\nu)g_{ij}-\frac{f}{(n-1)(n-2)}\nabla^2\varphi(\nu,\nu)g_{ij}\\
\nonumber&+\frac{f}{n-2}\nabla^2\varphi(E_{i},E_{j})+\frac{f^{2}}{m|\nabla f|}g(Rm(\nu,E_i)E_j,\nabla\varphi)-\frac{f^2}{m|\nabla f|}\frac{Ric(\nu,\nabla\varphi)}{n-1}g_{ij}\\
&+\frac{f}{n-1}\left\{\frac{\Delta\varphi}{m}+\frac{R+(m-n)\lambda}{m}+\frac{R-2(n-1)\lambda}{n-2}\right\}g_{ij}.
\end{align}
equation $\frac{m}{f}\Delta f=R+\Delta\varphi-n\lambda$ leads to the following identity
\begin{eqnarray*}
\frac{\Delta\varphi}{m}+\frac{R+(m-n)\lambda}{m}+\frac{R-2(n-1)\lambda}{n-2}&=&\frac{\Delta\varphi}{m}+\frac{n+m-2}{m(n-2)}
(R-n\lambda)\\
&=&\frac{\Delta\varphi}{n-2}+\frac{n+m-2}{n-2}\frac{\Delta f}{f}.
\end{eqnarray*}
Substituting this last expression in \eqref{SUBST} and reordering the terms we obtain
\begin{align*}
&\frac{n+m-2}{n-2}\nabla^2f(E_{i},E_{j})+\frac{n+m-2}{(n-1)(n-2)}\nabla^2f(\nu,\nu)g_{ij}-\frac{(n+m-2)\Delta f}{(n-1)(n-2)}g_{ij}\\
=&-\frac{f}{(n-1)(n-2)}\nabla^2\varphi(\nu,\nu)g_{ij}+\frac{f\Delta\varphi}{(n-1)(n-2)}g_{ij}-\frac{f}{n-2}\nabla^2\varphi(E_{i},E_{j})\\
&+\frac{f^{2}}{m|\nabla f|}g(Rm(\nu,E_i)E_j,\nabla\varphi)-\frac{f^2}{m|\nabla f|}\frac{Ric(\nu,\nabla\varphi)}{n-1}g_{ij}
\end{align*}
and the result follows.
\end{proof}

Besides the above results our Theorem \ref{MainThm22} also relies on the fundamental fact that: if a Riemannian manifold $(\Bbb{M}^n,g)$ admits a Codazzi tensor such that at every point of $\Bbb{M}^n$ it has exactly two distinct differentiable eigenvalue functions with multiplicity $1$ and $n-1$, then the tangent bundle splits as the orthogonal direct sum of two integrable distributions, a  $1$--dimensional totally geodesic and other $(n-1)$--dimensional totally umbilical (see Derdzinski \cite{Derdzinski} and Proposition 16.11 in \cite{besse} for further comments).

\begin{proposition}\label{MainThm2}
Let $(\Bbb{M}^n,g,\varphi,f,\lambda)$, $n\geq 3$, be a Ricci-Hessian type manifold with harmonic Weyl tensor, $W(\nabla f, \cdot, \cdot)$ be null and $\nabla\varphi$ normal on $f$--hypersurfaces.
In addition, suppose that $(\Bbb{M}^n,g)$ is $\varphi$--radially flat. Then, $(\Bbb{M}^n,g)$ is, around any regular point of $f$, locally isometric to a warped product with $(n-1)$--dimensional Einstein fiber.
\end{proposition}
\begin{proof}
For this proof we consider the orthonormal frame as well as the results obtained above. As $\nabla\varphi=h\nabla f$ we have
\begin{align}\label{eq-hess}
\nabla^2\varphi=h\nabla^2 f+dh\otimes df,
\end{align}
which implies
\begin{align*}
\nabla^2\varphi(E_i,E_j)=h\nabla^2 f(E_i,E_j)\quad\mbox{and}\quad \nabla^2\varphi(\nu,\nu)=h\nabla^2 f(\nu,\nu)+g(\nabla h,\nabla f).
\end{align*}
Therefore,
\begin{align*}
\Delta\varphi&=h\Delta f+g(\nabla h,\nabla f).
\end{align*}
On the other hand,
\begin{align*}
0=Rm(\cdot,\nabla\varphi)\nabla\varphi=h^2 Rm(\cdot,\nabla f)\nabla f \quad\mbox{and}\quad 0=Ric(\nabla\varphi,\nabla\varphi)=h^2 Ric(\nabla f,\nabla f).
\end{align*}
So, we have $(Rm(\cdot,\nabla f)\nabla f)(x)=0$ and $Ric(\nabla f,\nabla f)(x)=0$ for $x\in\Bbb{M}$ such that $h(x)\neq 0$. By continuity, the same is true in the points $y\in\Bbb{M}$ where $h(y)=0$, since $y\in\mathcal{Z}_\varphi$.

With this considerations in mind the relation \eqref{eqprop1} given by Proposition \ref{prop1} becomes
\begin{align*}
&(n+m-2)\{\nabla^2f(E_i,E_j)-\frac{1}{n-1}(\Delta f-\nabla^2f(\nu,\nu))g_{ij}\}\\
=&-fh\nabla^2 f(E_i,E_j)-\frac{fh}{n-1}(\Delta f-\nabla^2 f(\nu,\nu))g_{ij}.
\end{align*}
Hence,
\begin{align}\label{eqteo2}
(n+m-2+hf)\{\nabla^2f(E_i,E_j)-\frac{1}{n-1}(\Delta f-\nabla^2f(\nu,\nu))g_{ij}\}=0.
\end{align}
We will have two cases to consider:

\noindent\emph{\textbf{(a) First case:}}

\noindent\emph{\textbf{Step 1.}} Suppose that there exists a point $q$ of the $f$--hypersurface where $(n+m-2+hf)(q)\neq 0$. Then, from equation \eqref{eqteo2} and by continuity, there exist an open set $U\subset\Bbb{M}$ where occur
\begin{equation}\label{et1}
\nabla^2f(E_i,\cdot)=\sigma g(E_i,\cdot) \quad \mbox{for}\quad i\in\{1,\ldots,n-1\}\quad\mbox{and}\quad\sigma=\Delta f-\nabla^2f(\nu,\nu).
\end{equation}
Moreover, by Proposition \ref{MainThmEP}, $\nabla f$ is an eigenvector of $Ric$. Consequently, by equation \eqref{b2}, $\nabla f$ is also an eigenvector of $\nabla^2f$, which means that
\begin{equation}\label{et2}
\nabla^2f(\nabla f,\cdot)=\rho g(\nabla f,\cdot)
\end{equation}
for a smooth function $\rho$.

From the relations \eqref{et1} and \eqref{et2}
we conclude that $\sigma$ and $\rho$ are eigenvalues of $\nabla^2f$ with multiplicity $n-1$ and $1$, respectively. Using equation \eqref{eq-hess} and the fundamental equation we have
\begin{align}\label{eq-for-eigen}
Ric+(h-\frac{m}{f})\nabla^2f=\lambda g-dh\otimes df.
\end{align}
It easy to see that $\nabla f$ and $E_i$ are eigenvectors of the tensor $dh\otimes df$.
Thus, from \eqref{eq-for-eigen} we verify that the Ricci tensor will also admit two eigenvalues with the same multiplicities, where $\nabla f$ is the eigenvector associated to the eigenvalue of multiplicity one. Consequently, the same occur with the Schouten tensor given by \eqref{c2}. Moreover, the Schouten tensor is Codazzi because $(\Bbb{M},g)$ has harmonic Weyl tensor.

\noindent\emph{\textbf{Step 2.}} For each point $p\in U$ we have the decomposition $T_p\Bbb{M}=[\nabla f]_p\oplus[\nabla f]_p^\perp$, where the integral manifolds of $[\nabla f]_p^\perp$ are totally umbilical. Next, we consider a coordinates system $(x_1,\ldots,x_{n-1},x_n)$ with coordinate basis $\{\partial_1,\ldots,\partial_{n-1},\partial_n=\nu\}$ for which the metric $g$ restricted to each integral submanifold $\Bbb{F}\subset \Bbb{M}$ of $[\nabla f]_p^\perp$ satisfies
\begin{equation*}
\mathcal{A}(\partial_i,\partial_j)=Hg_{ij},
\end{equation*}
where $g_{ij}=g(\partial_i,\partial_j)$ for $i,j=1,\ldots,n-1$, whereas $\mathcal{A}$ and $H$ stands for the second fundamental form and the mean curvature of $\Bbb{F}$, respectively.

From Codazzi equation we have
\begin{eqnarray*}
\partial_i(H)g_{jk}-\partial_j(H)g_{ik}&=&(\nabla_{\partial_i}\mathcal{A})(\partial_j,\partial_k)-(\nabla_{\partial_j}\mathcal{A})(\partial_i,\partial_k) = g(Rm(\partial_i,\partial_j)\partial_k,\nu).
\end{eqnarray*}
Then,
\begin{equation*}
Ric(\nu,\partial_i)=(n-1)\partial_i(H)\quad \mbox{for all}\quad i\in\{1,\ldots,n-1\}.
\end{equation*}
As $\nu$ is eigenvector of $Ric$ we get $\partial_i(H)=0$. So, $H$ is constant on $\Bbb{F}$. Since $[\nu,\partial_i]=0$ we obtain
\begin{eqnarray}\label{EDO-g}
\nu (g_{ij})&=&g(\nabla_{\partial_i}\nu,\partial_j)+g(\nabla_{\partial_j}\nu,\partial_i) = -2\mathcal{A}(\partial_i,\partial_j) = -2H g_{ij}.
\end{eqnarray}
Since $H$ is a function of $x_n$ only, then we have the solution $g_{ij}=e^{-\theta}G_{ij}$ of \eqref{EDO-g}, where $\theta=\theta(x_n)$ and $G$ is another metric on $\Bbb{F}$. Therefore, in the neighborhood of $q$, $\Bbb{M}$ is of the form $I\times_{e^{-\theta}}\Bbb{F}$ with metric
\begin{equation*}
g=dx_n^2+e^{-\theta}G.
\end{equation*}
Since $(\Bbb{M},g)$ has harmonic Weyl tensor, we have that $(\Bbb{F},G)$ is a space of constant sectional curvature ( see for instance \cite{G}).

\noindent\emph{\textbf{(b) Second case:}} Suppose that there exist a point $q$ belonging to $f$--hypersurface such that $(n+m-2+fh)(q)=0$, then, in this point we have
\begin{align}\label{eq-h}
h=\frac{c}{f}<0 \quad \mbox{for}\quad  c=-(n+m-2).
\end{align}
By continuity, equation \eqref{eq-h} is also valid in an open set $U\subset\Bbb{M}$ containing $q$ and we have $\nabla\varphi=\frac{c}{f}\nabla f$ on this open set. Moreover, an easy computation shows that
\begin{equation*}
\nabla^2(\ln f)=\frac{1}{f}\nabla^2f-d(\ln f)\otimes d(\ln f).
\end{equation*}
Hence,
\begin{equation*}
\nabla^2\varphi=\frac{c}{f}\nabla^2f-\frac{c}{f^2}df\otimes df = c[\frac{1}{f}\nabla^2f-d(\ln f)\otimes d(\ln f)]=c\nabla^2(\ln f)
\end{equation*}
and the fundamental equation becomes
\begin{equation}
Ric_g-(m-c)\nabla^2(\ln f)-m d(\ln f)\otimes d(\ln f)=\lambda g.
\end{equation}
Since that $(\Bbb{M},g)$ has harmonic Weyl tensor, we can choose a smooth function $u$ and a metric $\bar{g}$ such that $\bar{g}=e^{2u}g$. By conformal theory, $Ric_{\bar{g}}$ is related to $Ric_{g}$ by the well-known formula (see for instance \cite{besse})
\begin{equation*}
Ric_{\bar{g}}=Ric_{g}-(n-2)\nabla^2u+(n-2)du\otimes du-[(n-2)|\nabla u|^2+\Delta u]g.
\end{equation*}
Taking $u=\frac{m-c}{n-2}\ln f$ we have
\begin{equation*}
Ric_{\bar{g}}=\lambda g+m\frac{(n-2)^2}{(m-c)^2} du\otimes du+(n-2)du\otimes du-[(n-2)|\nabla u|^2+\Delta u]g.
\end{equation*}
By regrouping
\begin{equation*}
Ric_{\bar{g}}=\left[m\frac{(n-2)^2}{(m-c)^2} +(n-2)\right]du\otimes du-[(n-2)|\nabla u|^2+\Delta u-\lambda]e^{-2u}\bar{g}.
\end{equation*}
Since both $du\otimes du$ and $\bar{g}$ admit eigenvalues of multiplicity $1$ (with eigenvector $\nabla u=\frac{m-c}{(n-2)f}\nabla f$) and $n-1$, we deduce that $Ric_{\bar{g}}$ also does. Then once again we have the Schouten tensor $S_{\bar{g}}$ with two eigenvalues of multiplicity $1$ and $n-1$.

The Cotton tensors $C_{\bar{g}}$ and $C_g$ can be related by formula
\begin{equation}\label{C-bar-g}
(n-2)C_{\bar{g}}=(n-2)C_g-W_g(\nabla u,\cdot,\cdot).
\end{equation}
See appendix in \cite{C} for a demonstration.
By hypothesis $\dv_g W=0$ and equation \eqref{divW} we have $C_g=0$. From \eqref{C-bar-g} we have $C_{\bar{g}}=0$ and $S_{\bar{g}}$ is a Codazzi tensor. Now, we proceed as in the \emph{Step 2} of the \emph{First case} to conclude the proof of the proposition.
\end{proof}

\begin{corollary}[\cite{C}]
Let $(\Bbb{M}^n,g,\psi,\lambda)$, $n\geq 3$, be a generalized $m$--quasi-Einstein manifold with harmonic Weyl tensor and $W(\nabla\psi,\cdot,\cdot)$ is identically null. Then, around any regular point of $\psi$, $(\Bbb{M}^n,g)$ is locally a warped product with $(n-1)$--dimensional Einstein fiber.
\end{corollary}
\begin{proof}
As we already explained in Remark \ref{MainRemark}, we always can convert a generalized $m$--quasi-Einstein manifold $(\Bbb{M},g,\psi,\lambda)$ to Ricci-Hessian type manifold $(\Bbb{M},g,\varphi,f,\lambda)$, such that, $\mathcal{Z}_\psi=\mathcal{Z}_\varphi=\mathcal{Z}_f$. So, since $\mathcal{Z}_\psi$ has zero measure, we can apply Proposition \ref{MainThm2} in order to obtain our corollary.
\end{proof}

\subsection{Proof of Theorem \ref{MainThm22}}
\begin{proof}
Taking into account Proposition \ref{MainThm2} and that the class of manifolds with harmonic Weyl tensor includes the class of the conformally flat manifolds, \emph{a fortiori}, holds Theorem \ref{MainThm22}.
\end{proof}

\section{Concluding remarks}\label{Concluding remarks}
As an immediate consequence of Proposition \ref{MainThmEP} we have: If $(\Bbb{M}^n,g,\varphi,f,\lambda)$, $n\geq3$, is a Ricci-Hessian type manifold with constant sectional curvature, then $\nabla f$ is an eigenvector of $\nabla^2\varphi$. Indeed, since $(\Bbb{M}^n,g)$ has constant sectional curvature, it is locally conformally flat and holds the identity $Ric=(n-1)Kg$, where $K$ is the sectional curvature of $\Bbb{M}$. Furthermore, we get
\begin{equation*}
\Omega(\nabla\varphi,\nabla f)=g(Rm(\cdot,\cdot)\nabla\varphi,\nabla f)=-K d\varphi\wedge df.
\end{equation*}
From Proposition \ref{MainThmEP} we obtain $\nabla^2\varphi(\nabla f,\cdot)=\gamma g(\nabla f,\cdot)$, for some smooth function $\gamma$ on $\Bbb{M}$. This proves our claim.

As an other consequence of Proposition \ref{MainThmEP} we can verify that, if $(\Bbb{M}^n,g,\psi,\lambda)$ is a locally conformally flat generalized $m$--quasi-Einstein manifold with $n\geq3$, then $\nabla\psi$ is an eigenvector of the Ricci operator. Indeed, by Proposition \ref{MainThmEP} we have $Ric(\nabla\psi,\cdot)\wedge d\psi=0$. Thus, in a neighborhood of any point of $\Bbb{M}$ there exists a smooth function $\eta$ such that $Ric(\nabla\psi,\cdot)=\eta d\psi$. Since $\mathcal{Z}_\psi$ has zero measure, this is sufficient to conclude that $Ric(\nabla \psi)=\eta\nabla\psi$ everywhere on $\Bbb{M}$. This fact is also verified in \cite{BGV} and \cite{C} making use of different procedure.

\end{document}